\documentclass[leqno,12pt]{amsart}
\usepackage{amsfonts}
\usepackage{amsmath,amssymb,amsthm}

\everymath{\displaystyle}

\newcommand{\E}{\mathbb E}
\newcommand{\R}{\mathbb R}
\newcommand{\tr}{\mathrm{tr}}
\newcommand{\ds}{\displaystyle}
\newcommand{\manifold}[1]{\mathcal{#1}}
\newcommand{\M}{\manifold{M}}

\renewcommand{\span}{\mathrm{span}}
\newcommand{\const}{\mathrm{const}}

\newtheorem{thm}{Theorem}[section]
\newtheorem{cor}[thm]{Corollary}

\newtheorem{prop}[thm]{Proposition}
\theoremstyle{definition}

\theoremstyle{remark}

\numberwithin{equation}{section}

\begin{document}

\title[Timelike Meridian Surfaces of elliptic type in the Minkowski 4-Space]{Timelike Meridian Surfaces of elliptic type in the Minkowski 4-Space}

\author{Victoria Bencheva, Velichka Milousheva}

\address{Institute of Mathematics and Informatics, Bulgarian Academy of Sciences,
Acad. G. Bonchev Str. bl. 8, 1113, Sofia, Bulgaria}
\email{viktoriq.bencheva@gmail.com}
\email{vmil@math.bas.bg}

\subjclass[2020]{Primary 53B30, Secondary 53A35, 53B25}
\keywords{Meridian surfaces, surfaces with constant Gauss curvature, CMC-surfaces, surfaces with parallel normalized mean curvature vector field, canonical isotropic parameters}

\begin{abstract}

We consider a special family of 2-dimensional timelike surfaces in the  Minkowski 4-space $\R^4_1$ which lie on rotational hypersurfaces with timelike axis and 
call them meridian surfaces of elliptic type. 
We study the following basic classes of timelike meridian surfaces of elliptic type: with constant Gauss curvature, with constant mean curvature, with parallel  mean curvature vector field, with parallel normalized mean curvature vector field. The  results obtained for the last class are used to give explicit solutions to the background systems of natural PDEs describing the timelike surfaces with parallel normalized mean curvature vector field in $\R^4_1$.

\end{abstract}

\maketitle

\section{Introduction}

In the local theory of surfaces both in Euclidean and pseudo-Euclidean spaces, one of the basic problems is to find a minimal number of functions, satisfying some natural conditions, that determine the surface up to a motion. This problem is known as the Lund-Regge problem \cite{Lund-Reg}. It is solved for zero mean curvature surfaces of co-dimension two in the Euclidean 4-space $\R^4$ in  
 \cite{Itoh,Trib-Guad}, for spacelike and timelike zero mean curvature surfaces in the Minkoswki space $\R^4_1$ in \cite{Al-Pal} and \cite{G-M-IJM}, respectively,  for spacelike and timelike zero mean curvature surfaces in the pseudo-Euclidean space with neutral metric $\R^4_2$ in \cite{Sa} and \cite{A-M-1}, respectively. The spacelike or timelike surfaces with zero mean curvature in these spaces admit (at least locally) special isothermal parameters, called \textit{canonical},  such that the two main invariants -- the Gaussian curvature and the normal curvature of the surface satisfy a system of two partial differential equations called a \textit{system of natural PDEs}. So, the number of the invariant functions determining the surfaces and the number of the differential equations are reduced to two. Moreover, the geometry of the corresponding zero mean curvature surface  is determined  by the solutions of this system of natural PDEs. 

Thus, another natural question  arises: can we solve the Lund-Regge problem for other classes of surfaces,  different from the minimal ones, in 4-dimensional spaces? 
Or, equivalently, can we introduce canonical parameters and obtain natural equations for other classes of surfaces,  different from the minimal ones, in 4-dimensional spaces? 
We managed to  solve this problem for the surfaces with parallel normalized mean curvature vector field -- another important class of surfaces both in Riemannian and pseudo-Riemannian geometry, since being a natural extension of the surfaces with parallel mean curvature vector field, they play an important role in Differential geometry and  Physics.
We proved that the surfaces with parallel normalized mean curvature vector field (PNMCVF) can be described  in terms of three  functions satisfying a system of three partial differential equations. 
The results for PNMCVF-surfaces in the Euclidean 4-space $\R^4$ and spacelike PNMCVF-surfaces in the Minkowski 4-space $\R^4_1$ are obtained in \cite{G-M-Fil}. A similar result is proved for one class of Lorentz PNMCVF-surfaces in the pseudo-Euclidean 4-space $\R^4_2$ (see \cite{AGM}), but for the rest of them the question is still open.
In \cite{B-M} we obtained the background systems of natural partial differential equations describing
the timelike surfaces with parallel normalized mean curvature vector field in the Minkowski space $\R^4_1$. Our approach to the study of PNMCVF-surfaces both in Euclidean and pseudo-Euclidean spaces  is based on the introducing of special geometric parameters on each  such surface which we call \textit{canonical parameters}. In the case of timelike PNMCVF-surfaces we use canonical isotropic parameters. 
 Examples of solutions to the background systems of PDEs can be found in the class of the so called meridian surfaces -- these are 2-dimensional surfaces lying on
rotational hypersurfaces in $\R^4$, $\R^4_1$ or $\R^4_2$.

The search for explicit solutions to the system of PDEs that describes timelike surfaces with parallel normalized mean curvature vector field in $\R^4_1$ motivated us to consider timelike meridian  surfaces in the Minkowski 4-space.

\vskip 1mm
In the present paper, we consider a special family of 2-dimensional timelike surfaces in the four-dimensional Minkowski space $\mathbb R^4_1$ which lie on rotational hypersurfaces with timelike axis.
We call them \textit{meridian surfaces of elliptic type}. 
We study some basic classes of timelike meridian surfaces of elliptic type: with constant Gauss curvature, with constant mean curvature, with parallel  mean curvature vector field, with parallel normalized mean curvature vector field. 
The class of timelike meridian surfaces of elliptic type  with parallel normalized mean curvature vector field is described in Theorem \ref{Th:meridSurfParallelNorm}. We use this result to give explicit solutions to the background systems of natural PDEs describing the timelike surfaces with parallel normalized mean curvature vector field in $\R^4_1$.
We introduce isotropic canonical parameters on any timelike meridian surface of elliptic type and obtain the geometric  isotropic frame field for this class of surfaces and the geometric functions introduced in   \cite{B-M}.
Finally, we give examples of explicit solutions to the background systems of natural PDEs describing the PNMCVF-surfaces in $\R^4_1$.

\section{Preliminaries}

We consider the four-dimensional Minkowski space  $\mathbb R^4_1$   endowed with the standard flat metric 
$\langle ., . \rangle$ of signature $(3,1)$ which is given in local coordinates by
$dx_1^2 + dx_2^2 + dx_3^2 -dx_4^2,$
where $\left( x_{1},x_{2},x_{3},x_{4}\right) $ is a rectangular coordinate
system of $\mathbb{R}^4_1$.  Since the metric is indefinite, 
 a vector $v \in \R^4_1$ can have one of the three casual characters: it can be \textit{spacelike} if $\langle v, v \rangle >0$
or $v=0$, \textit{timelike} if $\langle v, v \rangle<0$, and \textit{lightlike} if $\langle v, v \rangle =0$ and $v\neq 0$. The terminology is inspired by the General Relativity.

We use the following denotations:
\begin{equation*}
\begin{array}{l}
\vspace{2mm}
\mathbb{S}^3_1(1) =\left\{V\in \R^4_1: \langle V, V \rangle =1 \right\}; \\
\vspace{2mm}
\mathbb{H}^3_1(-1) =\left\{ V\in \R^4_1: \langle V, V \rangle = -1\right\}.
\end{array}
\end{equation*}
The space $\mathbb{S}^3_1(1)$ is known as the de Sitter space, and the
space $\mathbb{H}^3_1(-1)$ is known as  the hyperbolic space (or the anti-de Sitter space) \cite{O'N}.

Recall that a surface  $\M^2$ in $\mathbb R^4_1$ is said to be
\emph{spacelike} (resp. \emph{timelike}), if $\langle ., . \rangle$ induces  a Riemannian (resp. Lorentzian) 
metric $g$ on $\M^2$. In the present paper we study timelike surfaces, so at each point $p\in \M^2$ we have the following decomposition
$$\R^4_1 = T_p \M^2 \oplus N_p \M^2$$
with the property that the restriction of the metric
onto the tangent space $T_p \M^2$ is of
signature $(1,1)$, and the restriction of the metric onto the normal space $N_p \M^2$ is of signature $(1,1)$.

The Levi Civita connections on $\mathbb R^4_1$ and $\M^2$ will be denoted by $\widetilde{\nabla}$ and $\nabla$, respectively.
Thus, for any tangent vector fields $x$ and $y$ and any normal vector field $\xi$  we can write the following  Gauss and Weingarten formulas:
$$\begin{array}{l}
\vspace{2mm}
\widetilde{\nabla}_xy = \nabla_xy + \sigma(x,y);\\
\vspace{2mm}
\widetilde{\nabla}_x \xi = - A_{\xi} x + D_x \xi,
\end{array}$$
where $\sigma$ is the second fundamental tensor, $D$ is the normal connection, 
and  $A_{\xi}$ is the shape operator with respect to $\xi$. In general, $A_{\xi}$ is not diagonalizable.

The mean curvature vector  field $H$ of $\M^2$ is defined as
$$H = \ds{\frac{1}{2}\,  \tr\, \sigma}.$$ 
A surface is called \textit{minimal} if its
mean curvature vector vanishes identically, i.e.  $H=0$. A normal vector field $\xi$ on a surface $\M^2$ is called \emph{parallel in the normal bundle} (or simply \emph{parallel}) if $D{\xi}=0$ 
\cite{Chen}.
A surface $\M^2$ is said to have \emph{parallel mean curvature vector field} if its mean curvature vector $H$ is parallel, i.e.
$D H =0$.

Surfaces for which the mean curvature vector field $H$ is non-zero, i.e. $\langle
H, H \rangle \neq 0$, and  the unit vector field in the
direction of $H$  is parallel in the normal bundle, are called surfaces with \textit{parallel normalized mean
curvature vector field} \cite{Chen-MM}. It can easily be seen that if $M$ is a
surface with non-zero parallel mean curvature vector field $H$ (i.e. $DH = 0$%
), then $M$ is a surface with parallel normalized mean curvature vector
field, but the converse is not true in general. It is true only in the case $%
\Vert H \Vert = const$. In  \cite{Chen-MM}, it is proved that every analytic surface with parallel normalized mean curvature
vector  in the Euclidean $m$-space $\mathbb{R}^{m}$ must either lie in a
4-dimensional space $\mathbb{R}^{4}$ or in a hypersphere of $\mathbb{R}^{m}$
as a minimal surface.

Complete classification of biconservative surfaces with parallel normalized mean curvature vector field in $\R^4$ is given in \cite{Sen-Turg-JMAA} and biconservative $m$-dimensional submanifolds with parallel normalized mean curvature vector field in $\R^{n+2} $ are studied in \cite{Sen}. Spacelike submanifolds with parallel normalized mean curvature vector field in a general de Sitter space are studied in \cite{Shu}.
For timelike surfaces in the Minkowski 4-space $\E^4_1$ or the pseudo-Euclidean 4-space with neutral metric $\E^4_2$ there are very few results on surfaces with  parallel normalized mean
curvature vector field, but non-parallel mean curvature vector field.

\vskip 1mm
In \cite{B-M-2}, we studied timelike surfaces free of minimal points in the Minkowski 4-space $\R^4_1$ and for each such surface we introduced a
pseudo-orthonormal frame field $\{x,y,n_1,n_2\}$, which is geometrically determined by the two lightlike directions  in the tangent space of the surface and the  mean curvature vector field. This  pseudo-orthonormal frame field  is called a \textit{geometric frame field} of the surface. Writing the derivative formulas with respect to the geometric frame field and using the integrability conditions, we obtained a system of six functions satisfying some differential equations and  proved a Fundamental Bonnet-type theorem stating that, in the general case, these six functions determine the surface up to a motion in $\R^4_1$. 
Bellow, we present briefly the construction.

\vskip 1mm
 In \cite{Lar}, it is shown that for each timelike surface $\M^2$ in $\R^4_1$ locally there exists a coordinate system  $(u,v)$ such that  the metric tensor $g$ of $\M^2$ has the following form:
\begin{equation*} \label{E:Eq-g}
g= - f^2(u, v)(du\otimes dv + dv\otimes du),
\end{equation*}
$f(u, v)$ being a  positive function. We suppose that $z=z(u, v), (u, v) \in \mathcal{D}$, $\mathcal{D} \subset \mathbb R^2$, is such a local parametrization on $\M^2$. Let us denote 
$z_u=\frac{\partial z}{\partial u},  \; 	z_v=\frac{\partial z}{\partial v}$.
Then, the coefficients of the first fundamental form are
\begin{equation*}
E = \langle z_u, z_u \rangle = 0; \quad F = \langle z_u, z_v \rangle = - f^2(u, v); \quad G = \langle z_v, z_v \rangle = 0.
\end{equation*}
Since both $z_u$ and $z_v$ are lightlike (isotropic), i.e.  $\langle z_u, z_u \rangle = 0$ and $\langle z_v, z_v \rangle = 0$, the parameters $(u,v)$ are called \textit{isotropic parameters} of the surface. 

We consider the following pseudo-orthonormal tangent frame field  of $\M^2$: $x=\displaystyle{\frac{z_u}{f}}$, $y=\displaystyle{\frac{z_v}{f}}$. Obviously,
 $\langle x, x \rangle = 0$,  $\langle x, y \rangle = -1$, $\langle y, y \rangle = 0$, and hence, the mean curvature vector field $H$ is expressed as: 
$$H = - \sigma(x, y).$$

For surfaces free of minimal points, i.e.  $H \neq 0$ at all points, we can choose a unit normal vector field $n_1$ such that $H = \nu n_1$  for a smooth function $\nu = || H ||$ and a unit normal vector field $n_2$ such that $\{n_1, n_2\}$ is an orthonormal frame field of the normal bundle ($n_2$ is determined up to orientation). Then, for the second fundamental tensor $\sigma$ we have the following formulas :
\begin{equation*}
\begin{array}{l} 
\vspace{2mm}
\sigma (x,x) = \lambda_1 n_1 + \mu_1 n_2; \\
\vspace{2mm}
\sigma (x,y) = -\nu n_1; \\
\vspace{2mm}
\sigma (y,y) = \lambda_2 n_1 + \mu_2 n_2,
\end{array}
\end{equation*} 
where $\lambda_1, \mu_1, \lambda_2, \mu_2$ are  smooth functions defined by:
$$
\lambda_1 = \langle \widetilde{\nabla}_x x, n_1 \rangle; \quad  \mu_1 = \langle \widetilde{\nabla}_x x, n_2 \rangle; \quad 
\lambda_2 = \langle \widetilde{\nabla}_y y, n_1 \rangle; \quad  \mu_2 = \langle \widetilde{\nabla}_y y, n_2 \rangle.  
$$
Keeping  in mind that $x=\ds{\frac{z_u}{f}}$, $y=\ds{\frac{z_v}{f}}$ and denoting $\gamma_1 = \frac{f_u}{f^2} = x(\ln f)$, $\gamma_2 = \frac{f_v}{f^2} = y(\ln f)$, we obtain the following derivative formulas:
\begin{equation}\label{E:DerivFormIsotr}
\begin{array}{ll}
\vspace{2mm}
\widetilde{\nabla} _x x = \gamma_1 x \qquad\quad + \lambda_1 n_1 + \mu_1 n_2;  & \qquad \widetilde{\nabla}_x n_1 = -\nu x  + \lambda_1 y \quad\quad + \beta_1 n_2; \\
\vspace{2mm}
\widetilde{\nabla}_x y = \quad\quad -\gamma_1 y -\nu n_1; & \qquad \widetilde{\nabla}_y n_1 = \lambda_2 x -\nu y \quad\quad\,\,\,\,\, + \beta_2 n_2; \\
\vspace{2mm}
\widetilde{\nabla}_y x = -\gamma_2 x \quad\quad -\nu n_1; & \qquad  \widetilde{\nabla}_x n_2 = \quad\quad + \mu_1 y  -\beta_1 n_1; \\
\vspace{2mm}
\widetilde{\nabla}_y y = \quad\quad\,\,\, \gamma_2 y \, \, +  \lambda_2 n_1 + \mu_2 n_2; & \qquad \widetilde{\nabla}_y n_2 = \mu_2 x \quad\quad\,\,\,\, -  \beta_2 n_1,
\end{array}
\end{equation}
where $\beta_1 = \langle \widetilde{\nabla}_x n_1 , n_2 \rangle$ and  $\beta_2 = \langle \widetilde{\nabla}_y n_1 , n_2 \rangle$. 

Formulas \eqref{E:DerivFormIsotr} are the derivative formulas of the surface with respect to the geometric frame field $\{x,y,n_1,n_2\}$. 
The functions $\gamma_1$, $\gamma_2$, $\nu$, $\lambda_1$, $\mu_1$, $\lambda_2$, $\mu_2$, $\beta_1$, $\beta_2$ are called geometric functions of the surface since they are determined by the geometric frame field.
In the general case  $\mu_1 \mu_2 \neq 0$, the  functions $f$, $\nu$, $\lambda_1$, $\mu_1$, $\lambda_2$, $\mu_2$ determine the surface up to a motion in $\R^4_1$ (\cite{B-M-2}, Theorem 3.2). 

The geometric meaning of the functions $\beta_1$ and  $\beta_2$ is connected with the class of surfaces with parallel mean curvature vector field. More precisely, in \cite{B-M}, we proved that: 

\begin{itemize}
\item
$\M^2$ has parallel mean curvature vector field if and only if $\beta_1 = \beta _2 =0$ and $\nu =const$.
\item
 $\M^2$ has parallel normalized mean curvature vector field if and only if $\beta_1 = \beta _2 =0$ and $\nu \neq const$.
\end{itemize}

In \cite{B-M} we showed that by introducing special  so-called canonical isotropic parameters on each surface with parallel normalized mean curvature vector field, we reduce up to three the number of functions determining this class of surfaces. To be precise, we proved that the timelike surfaces with parallel normalized mean curvature vector field in $\R^4_1$ for which $K - H^2 > 0$ (resp. $K - H^2 < 0$)  are determined up to a rigid motion in $\R^4_1$ by three functions 
  $\lambda(u,v)$, $\mu(u,v)$, and $\nu(u,v)$, $\mu  \neq 0$, $\nu \neq const$, satisfying the following  system of partial differential equations: 
\begin{equation} \label{E:Eq-fund-1}
\begin{array}{l}
\vspace{2mm}
\nu_u + \lambda_v = \lambda (\ln|\mu|)_v;\\
\vspace{2mm}
\lambda_u - \varepsilon \nu_v = \lambda (\ln|\mu|)_u;\\
\vspace{2mm}
|\mu| \left(\ln |\mu|\right)_{uv} = - \nu^2 - \varepsilon (\lambda^2 + \mu^2),
\end{array} 
\end{equation}
where  $\varepsilon =  1$ in the case $K - H^2 >0$ (resp. $\varepsilon = - 1$ in the case $K - H^2 <0$).

 The surfaces characterized by $K - H^2 =0$ are determined up to a rigid motion by three functions $\lambda(u,v)$, $\mu(u,v)$, and $\nu(u)$ 
satisfying \cite{B-M}:
\begin{equation}  \label{E:Eq0-5}
\begin{array}{l}
\vspace{2mm}
\nu_u + \lambda_v = \lambda (\ln|\mu|)_v;\\
\vspace{2mm}
|\mu| (\ln |\mu|)_{uv} = -\nu^2.
\end{array} 
\end{equation}

The above systems \eqref{E:Eq-fund-1} and \eqref{E:Eq0-5} are the background systems of natural partial differential equations describing
the timelike surfaces with parallel normalized mean curvature vector field in $\R^4_1$. Examples of solutions to these background systems of PDEs can be found considering the class of the so-called meridian surfaces in the Minkowski 4-space.

In the next section, we give the construction of timelike meridian surfaces of elliptic type in $\R^4_1$.

\section{Timelike meridian surfaces of elliptic type}

In the Minkowski 4-space $\R^4_1$ there are three types of rotational hypersurfaces: rotational hypersurfaces with timelike, spacelike, or
lightlike axis. Depending on the type of the spheres in $\R^3_1 \subset \R^4_1$ and $\R^3 \subset \R^4_1$ and the casual character of the spherical curves, we distinguish different types of timelike and spacelike meridian surfaces in $\R^4_1$. In this paper we will construct timelike meridian surfaces of elliptic type which are one-parameter systems of meridians of a rotational hypersurface with timelike
axis. Timelike meridian surfaces of hyperbolic and parabolic type will be studied separately.

Let $Oe_1 e_2 e_3 e_4$ be the standard orthonormal frame  in $\R^4_1$, i.e. $\langle e_1,e_1 \rangle =\langle e_2,e_2 \rangle = \langle e_3,e_3 \rangle = 1$, $\langle e_4,e_4 \rangle = -1$. We consider a
rotational hypersurface with timelike axis $Oe_4$, which is obtained as follows. 
Consider the  sphere  $\mathbb{S}^2(1) =\left\{V\in \R^3: \langle V, V \rangle = 1\right \}$ in the  Euclidean  3-space $\R^3 = \span\left \{ e_{1},e_{2},e_{3}\right \}$ centered at the origin. It is parametrized by
$$l(w^1,w^2) = \cos w^1 \cos w^2 \,e_1 + \cos w^1 \sin w^2 \,e_2 + \sin w^1 \,e_3.$$
Let $f = f(u), \,\, g = g(u)$ be smooth functions, defined in an
interval $I \subset \R$, such that $\dot{f}^2(u) - \dot{g}^2(u) \neq 0, \, f(u)>0,
\,\, u \in I$. 
 The rotational hypersurface $\mathcal{M}^3$ in $\R^4_1$, obtained by
the rotation of the meridian curve $m: u \rightarrow (f(u), g(u))$
about the $Oe_4$-axis,  has the following parametrization:
$$\mathcal{M}^3: Z(u,w^1,w^2) = f(u)\left( \cos w^1 \cos w^2 e_1 +  \cos w^1 \sin w^2 e_2 +  \sin w^1 e_3\right) + g(u) e_4.$$
Let $w^1 = w^1(v)$, $w^2=w^2(v), \,\, v \in J, \, J \subset \R$. We consider the two-dimensional
 surface $\mathcal{M}_m$ lying on $\mathcal{M}^3$, which is parametrized as follows:
\begin{equation}  \notag
\mathcal{M}_m: z(u,v) = Z(u,w^1(v),w^2(v)), \quad u \in I, \, v \in J.
\end{equation}
Obviously, $\mathcal{M}_m$ is a one-parameter system of meridians of the rotational hypersurface $\mathcal{M}^3$. For this reason, we call
$\mathcal{M}_m$ a  \emph{meridian surface of elliptic type}.

Note that $\textbf{c}: l = l(v) = l(w^1(v),w^2(v))$ is a smooth curve lying on the two dimensional sphere $S^2(1)$. The parametrization of the meridian surface $\mathcal{M}_m$ can be written as follows:
\begin{equation} \label{Eq:meridian_surf}
\mathcal{M}_m: z(u,v) = f(u) \, l(v) + g(u)\, e_4, \quad u \in I, \, v \in J.
\end{equation}

Spacelike meridian surfaces of elliptic type in the Minkowski 4-space were defined in  \cite{GM6} where a local classification of marginally trapped meridian surfaces was given. 
Spacelike meridian surfaces in $\R^4_1$ with pointwise 1-type Gauss map were classified in \cite{AM}. The classification of  spacelike meridian surfaces of elliptic  or hyperbolic type with constant Gauss curvature or with constant mean curvature was given in \cite{GM-MC}.

In this paper we will study the timelike case of meridian surfaces of elliptic type.

For convenience, when differentiating with respect to the parameter $u$ we use the notation $\dot{f}(u) = \ds \frac{\partial f}{\partial u}$ and when differentiating with respect to the parameter $v$ we use $l'(v) = \ds \frac{\partial l}{\partial v}$.

Since we are interested in timelike surfaces, we consider the case $\dot{f}^2(u) - \dot{g}^2(u) <0, \,\, u \in I$. Without loss of generality we may assume that $\dot{f}(u)^2-\dot{g}(u)^2 = -1$. Hence, it follows that  $\dot{f}\ddot{f}-\dot{g}\ddot{g} = 0$. The curvature $\varkappa_m$ of the meridian curve $m$ is given by
$\varkappa_m = \dot{f}\ddot{g}-\dot{g}\ddot{f}$.

Without loss of generality we assume that  the curve $\textbf{c}: l = l(v)$ lying on the sphere $S^2(1)$ is parametrized by the arc-length, i.e. for 
$l(v)$ we have: 
$$
\langle l(v), l(v) \rangle = 1, \quad \langle l'(v), l'(v) \rangle = 1.
$$
The Frenet formulas for the curve $\textbf{c}$ on $S^2(1)$ are given bellow:
$$
\begin{array}{l}
\vspace{2mm}
l' = t; \\
\vspace{2mm}
t'=  \varkappa n - l; \\
\vspace{2mm}
n'= - \varkappa t,
\end{array}
$$
where $\varkappa = \varkappa(v)$ is the spherical curvature of $\textbf{c}$, i.e. $\varkappa (v)= \langle t'(v), n(v) \rangle$, and $\{l(v), t(v), n(v)\}$ is an orthonormal frame field in $\R^3 = \span \{e_1, e_2, e_3\}$.

Then, the tangent vector fields $z_u$ and $z_v$ of $\M_m $ are expressed as follows:
$$
\begin{array}{l}
\vspace{2mm}
z_u = \dot{f}(u) l(v) + \dot{g}(u) e_4;\\
\vspace{2mm}
z_v =  f(u) t(v).
\end{array}
$$
Calculating the coefficients of first fundamental form of $\M_m $ we get:
$$
\begin{array}{l}
\vspace{2mm}
E = \langle z_u, z_u \rangle = \dot{f}(u)^2-\dot{g}(u)^2 = -1;\\
\vspace{2mm}
F = \langle z_u, z_v \rangle = 0;\\
\vspace{2mm}
G = \langle z_v, z_v \rangle = f^2(u).
\end{array}
$$
Hence,  $EG-F^2 = -f^2(u)$. 
We consider the following orthonormal tangent frame field:
\begin{equation} \label{E:Eq-mer-tangent}
\begin{array}{l}
\vspace{2mm}
X = z_u;  \\
\vspace{2mm}
Y = \ds \frac{z_v}{f} = t.
\end{array}
\end{equation}
Therefore, $\langle X,X \rangle = -1, \,\,  \langle X,Y \rangle = 0, \,\,  \langle Y,Y \rangle = 1$. 
We consider the following orthonormal frame field of the  normal space:
\begin{equation} \label{E:Eq-mer-normal}
\begin{array}{l}
\vspace{2mm}
N_1 = n(v); \\
\vspace{2mm}
N_2 =  \dot{g}(u) l(v) + \dot{f}(u) e_4.
\end{array}
\end{equation}
Obviously, 
$\langle N_1,N_1 \rangle = 1, \,\,  \langle N_1,N_2 \rangle = 0, \,\,  \langle N_2,N_2 \rangle = 1$. 
Calculating the second derivatives of the vector function $z(u,v)$ we obtain:
\begin{equation} \label{E:Eq-mer-second-deriv}
\begin{array}{l}
\vspace{2mm}
z_{uu} = \ddot{f} l + \ddot{g} e_4; \\
\vspace{2mm}
z_{uv} = \dot{f} l'(v) = \dot{f} t; \\
\vspace{2mm}
z_{vv} = f t' = f \varkappa n - f l, 
\end{array}
\end{equation}
which imply that 
$$
\begin{array}{ll}
\vspace{2mm}
\langle z_{uu}, N_1 \rangle = 0; & \qquad  \langle z_{uu}, N_2 \rangle = -\varkappa_m;\\
\vspace{2mm}
 \langle z_{uv}, N_1 \rangle = 0; & \qquad  \langle z_{uv}, N_2 \rangle = 0;\\
\vspace{2mm}
\langle z_{vv}, N_1 \rangle = f \varkappa; & \qquad  \langle z_{vv}, N_2 \rangle = - f \dot{g}.
\end{array}
$$

Using that  $X = z_u$, $Y = \ds \frac{z_v}{f}$ and formulas \eqref{E:Eq-mer-second-deriv}, we  obtain the following  derivative formulas for the frame field $\{X, Y, N_1, N_2\}$:
$$
\begin{array}{l}
\vspace{2mm}
\widetilde{\nabla}_X X = \ddot{f} l + \ddot{g} e_4;\\
\vspace{2mm}
\widetilde{\nabla}_X Y = 0;\\
\vspace{2mm}
\widetilde{\nabla}_Y X = \ds \frac{\dot{f}}{f} t;\\
\vspace{2mm}
\widetilde{\nabla}_Y Y = \ds \frac{\varkappa}{f} n - \ds \frac{1}{f} l;
\end{array} \qquad \quad
\begin{array}{l}
\vspace{2mm}
\widetilde{\nabla}_X N_1 = 0;\\
\vspace{2mm}
\widetilde{\nabla}_Y N_1 = \ds -\frac{\varkappa}{f} t;\\
\vspace{2mm}
\widetilde{\nabla}_X N_2 = \ddot{g} l + \ddot{f} e_4; \\
\vspace{2mm}
\widetilde{\nabla}_Y N_2 = \ds \frac{\dot{g}}{f} t.
\end{array} 
$$
Now, having in mind  \eqref{E:Eq-mer-tangent} and \eqref{E:Eq-mer-normal}, the formulas above imply:
\begin{equation} \label{E:Eq-mer-derivative-1}
\begin{array}{ll}
\vspace{2mm}
\widetilde{\nabla}_X X = \quad \quad \qquad \qquad - \varkappa_m n_2; & \qquad \widetilde{\nabla}_X N_1 = 0;\\
\vspace{2mm}
\widetilde{\nabla}_X Y = 0; & \qquad \widetilde{\nabla}_Y N_1 = \quad \quad \ds -\frac{\varkappa}{f} y;\\
\vspace{2mm}
\widetilde{\nabla}_Y X = \qquad \ds \frac{\dot{f}}{f}y; & \qquad \widetilde{\nabla}_X N_2 = - \varkappa_m x; \\
\vspace{2mm}
\widetilde{\nabla}_Y Y = \ds \frac{\dot{f}}{f} x  \quad +\frac{\varkappa}{f} n_1 - \frac{\dot{g}}{f} n_2; & \qquad \widetilde{\nabla}_Y N_2 =  \quad \qquad \ds \frac{\dot{g}}{f} y. 
\end{array} 
\end{equation}
Equations \eqref{E:Eq-mer-derivative-1} are the derivative formulas of the surface with respect to the parameters $(u,v)$.
From these formulas we can calculate the Gauss curvature $K$, the curvature of the normal connection $K^{\bot}$, and the mean curvature vector field $H$  of the meridian surface $\M_m$. 

The Gauss curvature $K$ is defined by the following formula:
$$
K = \frac{\langle \sigma(X,X) , \sigma(Y,Y)\rangle - \langle \sigma(X,Y) , \sigma(X,Y)\rangle}{\langle X,X\rangle \langle Y,Y\rangle - \langle X,Y\rangle ^2}.
$$
Hence, using \eqref{E:Eq-mer-derivative-1}, we get that the  Gauss curvature of $\M_m$ is expressed as follows: 
\begin{equation}\label{Eq:meridianSurfGaussCurv}
K = \frac{\ddot{f}(u)}{f(u)}.
\end{equation}
The above formula shows that the Gauss curvature depends only on the parameter $u$ since it is expressed in terms of the function $f(u)$ determining the meridian curve $m$.
 
The curvature of the normal connection (also called normal curvature) is defined by:
$$
K^{\bot} = \frac{\langle R^D(X,Y)N_1, N_2\rangle}{\langle X,X\rangle \langle Y,Y\rangle - \langle X,Y\rangle ^2},
$$
where 
$$R^D(X,Y)N = D_X D_Y N - D_Y D_X N - D_{[X,Y]}N.$$
Hence, by use of \eqref{E:Eq-mer-derivative-1}, we get that the  curvature of the normal connection of the meridian surface $\M_m$ is:
$$K^{\bot}=0.$$
The above formula implies the next statement.

\begin{prop}
Each timelike meridian surface $\M_m$ of elliptic type, defined by \eqref{Eq:meridian_surf}, is a surface with flat normal connection.
\end{prop}

Using \eqref{E:Eq-mer-derivative-1} we find that the normal mean curvature vector field $H$  depends on both parameters $u$ and $v$ and is given by the formula:
\begin{equation} \label{E:Eq-mer-H}
H = \frac{\varkappa (v)}{2f(u)} \, n_1 - \frac{1 + \dot{f}^2(u) + f(u) \ddot{f}(u)}{2f(u) \sqrt{\dot{f}^2(u) + 1}} \, n_2.
\end{equation}

\section{Timelike meridian surfaces of elliptic type with constant Gauss curvature}

In this section we describe all timelike meridian surfaces of elliptic type in $\R^4_1$ with constant Gauss curvature. First we consider the case when the Gauss curvature is zero, i.e. the surface is flat.

The flat timelike meridian surfaces of elliptic type are characterized in the following theorem.

\begin{thm}\label{Th:meridSurfFlat} Let $\M_m$ be a timelike meridian surface of elliptic type, defined by \eqref{Eq:meridian_surf}. Then, $\M_m$ is a flat surface if and only if the meridian curve $m$ is given by:
$$
\begin{array}{l}
\vspace{2mm}
f(u) = a \,u + b; \quad a = \const, \, b =\const;\\
\vspace{2mm}
g(u) = \pm \sqrt{a^2+1} \,u + c; \quad c = \const.
\end{array}
$$
\end{thm}

\begin{proof} 
If $\M_m$ is a timelike meridian surface, defined by \eqref{Eq:meridian_surf}, the Gauss curvature is expressed by formula \eqref{Eq:meridianSurfGaussCurv}. 
Hence, $K=0$ if and only if $\ddot{f}(u)=0$. The last equality implies that the function $f(u)$ is given by 
\begin{equation} \label{Eq:flat}
f(u) = a \,u + b
\end{equation}
 for some constants $a$ and $b$. 
Using that the function $g(u)$ is determined by $\dot{g}(u) = \pm \sqrt{\dot{f}^2(u)+1}$, from \eqref{Eq:flat} we obtain 
$$g(u) = \pm \sqrt{a^2+1} \,u + c, $$
where $c = \const$.
\end{proof}

In the next theorem  we describe all timelike meridian surfaces of elliptic type with constant non-zero Gauss curvature.

\begin{thm}\label{Th:meridSurfConstK} Let $\M_m$ be a timelike meridian surface of elliptic type,  defined by \eqref{Eq:meridian_surf}. Then, $\M_m$ has constant non-zero Gauss curvature $K$ if and only if the meridian curve $m$ is given by:
$$
\begin{array}{l}
\vspace{2mm}
f(u) = a_1 \cos \sqrt{-K} \,u + a_2 \sin \sqrt{-K} \, u,\quad \text{if} \,\, K < 0;\\
\vspace{2mm}
f(u) = a_1 \cosh \sqrt{K} \,u + a_2 \sinh \sqrt{K}\, u,\quad \text{if} \,\, K > 0,
\end{array}
$$
where $a_1$ and $a_2$ are constants and the function $g(u)$ is determined by $\dot{g}(u) = \pm \sqrt{\dot{f}^2(u)+1}$.
\end{thm}

\begin{proof}
Let $\M_m$ be a timelike meridian surface, defined by \eqref{Eq:meridian_surf}. Using  \eqref{Eq:meridianSurfGaussCurv}, we get that the Gauss curvature $K$ is constant if and only if the function $f(u)$ satisfies the following differential equation:
$$
\ddot{f}(u) - K f(u) = 0.
$$
The general solution of the equation above is given by:
$$
\begin{array}{l}
\vspace{2mm}
f(u) = a_1 \cos \sqrt{-K} \,u + a_2 \sin \sqrt{-K} \, u,\quad \text{in the case} \,\, K < 0;\\
\vspace{2mm}
f(u) = a_1 \cosh \sqrt{K} \,u + a_2 \sinh \sqrt{K}\, u,\quad \text{in the case} \,\, K > 0,
\end{array}
$$
where $a_1$ and $a_2$ are constants. The function $g(u)$ is determined by $\dot{g}(u) = \pm \sqrt{\dot{f}^2(u)+1}$.
\end{proof}

\section{Timelike meridian surfaces of elliptic type with constant mean curvature}

In this section we give a classification of timelike meridian surfaces of elliptic type in $\R^4_1$ with  constant mean curvature vector field.
First we consider the minimal case, i.e. $H = 0$. 

The classification of all minimal timelike meridian surfaces of elliptic type is given in the next theorem.

\begin{thm}\label{Th:meridSurfMin} 
Let $\M_m$ be a timelike meridian surface of elliptic type, defined by \eqref{Eq:meridian_surf}. Then, $\M_m$ is minimal if and only if the curve $\textbf{c}$ on $S^2(1)$ has zero spherical curvature and the meridian curve $m$ is given by
$$\begin{array}{l}
\vspace{2mm}
f(u) = \pm \sqrt{-u^2 +2au+b};\\
\vspace{2mm}
g(u) = \pm \sqrt{a^2+b} \arcsin \frac{u-a}{\sqrt{a^2+b}} + c,
\end{array}
$$
where $a = const$, $b=const$, $c = const$. 
\end{thm}

\begin{proof} 
Let $\M_m$ be a timelike meridian surface of elliptic type, defined by \eqref{Eq:meridian_surf}. Then, the mean curvature vector field is expressed by formula \eqref{E:Eq-mer-H}, and hence the surface 
 is minimal if and only if the spherical curvature of $\textbf{c}$ is $\varkappa = 0$ and the function $f(u)$ satisfies the following equation 
$$1 + \dot{f}^2 + f \ddot{f} = 0.$$
The solutions of this differential equation are expressed by the following formula:
$$
f(u) = \pm \sqrt{-u^2+2a u + b}, \quad a= \const, \, b = \const.
$$
Using that $\dot{g} = \pm\sqrt{\dot{f}^2+1}$, we get  the following equation for $g(u)$:
$$
\dot{g}= \pm \frac{\sqrt{a^2+b}}{\sqrt{-u^2+2a u + b}},
$$
Integrating the above equation we obtain
$$g(u) = \pm \sqrt{a^2+b} \arcsin \frac{u-a}{\sqrt{a^2+b}} + c, \quad c = const.$$

\end{proof}

\begin{cor}
There are no minimal timelike meridian surfaces of elliptic type in  $\R^4_1$ other than surfaces lying in a hyperplane of $\R^4_1$.
\end{cor}

\begin{proof}
In the case  $\M_m$  is minimal, the  spherical curvature of $\textbf{c}$ is $\varkappa = 0$. Hence,  from \eqref{E:Eq-mer-derivative-1} it follows that $\widetilde{\nabla}_X N_1 = 0, \; \widetilde{\nabla}_Y N_1 = 0$, which  means that the normal vector field $N_1$ is constant. So,  the meridian surface $\M_m$ lies in the constant 3-dimensional space $\R^3_1 = \span \{X,Y,N_2\}$. Consequently, $\M_m$ lies in a hyperplane of $\R^4_1$.

\end{proof}

In the next theorem we classify all timelike meridian surfaces with non-zero constant mean curvature (CMC-surfaces).

\begin{thm}\label{Th:meridSurfConstH} 
Let $\M_m$ be a timelike meridian surface of elliptic type,  defined by \eqref{Eq:meridian_surf}. Then, $\M_m$ has constant mean curvature vector field, i.e. $|| H || = a = \const, a \neq 0$,  if and only if the curve $\textbf{c}$ on $S^2(1)$ has constant spherical curvature $\varkappa = \const = b, b\neq 0$ and the meridian curve $m$ is given by $\dot{f} = \varphi (f)$, where
$$
\varphi(t) = \pm  \sqrt{ \frac{1}{t^2} \left(c \pm \frac{t}{2}  \sqrt{4 a^2 t^2-b^2} \mp \frac{b^2}{4a} \ln | 2a t +   \sqrt{4 a^2 t^2-b^2}| \right)^2 -1}, \quad c = \const,
$$
the function $g(u)$ is defined by $\dot{g} = \pm\sqrt{\dot{f}^2+1}$.
\end{thm}

\begin{proof} 
Let $\M_m$ be a timelike meridian surface of elliptic type,  defined by \eqref{Eq:meridian_surf}. Then, the mean curvature vector field $H$ is given by formula  \eqref{E:Eq-mer-H}, which implies that  $|| H || = \const = a$ if and only if the following equality holds:
$$
\varkappa^2 = \frac{4a^2 f^2  (\dot{f}^2 + 1) - (1 + \dot{f}^2 + f \ddot{f})^2}{\dot{f}^2 + 1}.
$$
Using that the left-hand side of the above equality is a function of the parameter $v$ and the right-hand side is a function of $u$, we conclude:
\begin{equation}\label{Eq:meridianSurfMeanCurvConstant}
\begin{array}{l}
\vspace{2mm}
\varkappa = \const = b, \,\, b \neq 0; \\
\vspace{2mm}
4a^2 f^2  (\dot{f}^2 + 1) - (1 + \dot{f}^2 + f \ddot{f})^2 = b^2 (\dot{f}^2 + 1).
\end{array}
\end{equation}
From the first equality of \eqref{Eq:meridianSurfMeanCurvConstant} it follows that the spherical curve $\textbf{c}$ has constant spherical curvature $\varkappa = b$, i.e. $\textbf{c}$ is a circle on $S^2(1)$. From the second equality of \eqref{Eq:meridianSurfMeanCurvConstant} we get the following differential equation:
\begin{equation}\label{Eq:meridianSurfMeanCurvConstantDifEq}
(1 + \dot{f}^2 + f \ddot{f})^2 = (\dot{f}^2 + 1)(4a^2 f^2 - b^2).
\end{equation}
Setting $\dot{f} = \varphi(f)$ in equation  \eqref{Eq:meridianSurfMeanCurvConstantDifEq}, we get that the function $\varphi = \varphi(t)$ is a solution to the following differential equation:
$$
1 + \varphi^2 + \frac{t}{2} (\varphi^2)' = \pm \sqrt{\varphi^2 + 1}\sqrt{4a^2 t^2 - b^2}.
$$
The general solution of the equation above is given by the formula:
\begin{equation}\label{Eq:meridianSurfMeanCurvSol}
\varphi(t) = \pm  \sqrt{ \frac{1}{t^2} \left(c \pm \frac{t}{2}  \sqrt{4 a^2 t^2-b^2} \mp \frac{b^2}{4a} \ln | 2a t +   \sqrt{4 a^2 t^2-b^2}| \right)^2 -1}, 
\end{equation}
where $c = \const$. The function $f$ is determined by $\dot{f} = \varphi(f)$ and \eqref{Eq:meridianSurfMeanCurvSol}, and the function $g$ is defined by $\dot{g} = \pm \sqrt{\dot{f}^2+1}$.

\end{proof}

\section{Timelike meridian surfaces of elliptic type with parallel mean curvature vector field}

Let $\M_m$ be a  timelike meridian surface of elliptic type. Then its mean curvature vector field $H$ is expressed by  \eqref{E:Eq-mer-H}. We consider the non-minimal case, i.e. we assume that $H \neq 0$.
The mean curvature vector field  $H$ is parallel in the normal bundle if and only if $D_X H = D_Y H = 0$.  
Using \eqref{E:Eq-mer-derivative-1} we calculate $D_X H$ and $D_Y H$ and obtain the following expressions:
\begin{equation}\label{Eq:MeanCurvDeriv}
\begin{array}{l}
\vspace{.2cm}
D_X H = - \frac{\varkappa \dot{f}}{2f^2} \,n_1 - \frac{\partial}{\partial u} \left(\frac{1 + \dot{f}^2 + f \ddot{f}}{2f \sqrt{\dot{f}^2 + 1}}\right)\,n_2; \\
\vspace{.2cm} 
D_Y H = \frac{\varkappa'}{2f^2} \,n_1.
\end{array}
\end{equation}

In the next theorem we describe all timelike meridian surfaces of elliptic type with parallel mean curvature vector field.

\begin{thm}\label{Th:meridSurfParralelH}
Let $\M_m$ be a timelike  meridian surface of elliptic type, defined by \eqref{Eq:meridian_surf}. Then,  $\M_m$ has parallel mean curvature vector field, if and only if one of the following cases holds:

\hskip 10mm (i)  $\textbf{c}$ has zero spherical curvature and the meridian 
$m$ is determined by $\dot{f} = \varphi(f)$ where 
\begin{equation}
\varphi(t) = \pm \frac{1}{t} \sqrt{(c + a\, t^2)^2 -t^2}, \quad a = const
\neq 0, \quad c = const,  \notag
\end{equation}
$g(u)$ is defined by $\dot{g} = \pm \sqrt{\dot{f}^2+1}$. In this case, $\M_m$ is a non-flat CMC-surface lying in a hyperplane of $\mathbb{E}^4_1$.

\hskip 10mm (ii)  $\textbf{c}$ has non-zero constant spherical curvature and the
meridian $m$ is determined by $f(u) = a$, $g(u) = \pm u + b$, where $a = const
\neq 0$, $b=const$. In this case, $\M_m$ is a flat CMC-surface lying in
a hyperplane of $\mathbb{E}^4_1$.
\end{thm}

\begin{proof}

Let $\M_m$ be a timelike  meridian surfaces of elliptic type with parallel mean curvature
vector field. Having in mind  formulas \eqref{Eq:MeanCurvDeriv}, we get the following conditions 
\begin{equation}  \label{E:Eq-7}
\begin{array}{l}
\vspace{2mm} \varkappa'(v) = 0; \\ 
\vspace{2mm} \varkappa \dot{f} = 0; \\ 
\vspace{2mm} \frac{{f \ddot{f}}+ \dot{f}^2 + 1}{2f \sqrt{\dot{f}^2 + 1}} =const.%
\end{array}%
\end{equation}
The first equality of \eqref{E:Eq-7} implies that the spherical curvature $\varkappa$ of the curve $\textbf{c}$ is constant. From the second equality of \eqref{E:Eq-7} we obtain that
there are two possible cases:

\vskip1mm 
Case (i): $\varkappa =0$, which means that the curve $\textbf{c}: l = l(v)$ is a great circle on $S^2(1)$. Moreover, 
$\widetilde{\nabla}_X N_1 = 0, \; \widetilde{\nabla}_Y N_1 = 0$, which  implies $N_1 = const$. So,  the surface $\M_m$ lies in the constant 3-dimensional space $\R^3_1 = \span \{X,Y,N_2\}$. 
From the third equality of \eqref{E:Eq-7} we get that the function $f(u)$  is a solution to the following differential equation: 
$$f \ddot{f}+ \dot{f}^2 + 1 = a \,2f \sqrt{\dot{f}^2 + 1},$$ 
where $a=const$. Since we consider non-minimal surfaces, we assume that $a\neq 0$. 
To find the solutions of the above differential equation we set $\dot{f}=\varphi (f)$ and obtain that the function $\varphi =\varphi (t)$ is a solution to the equation: 
\begin{equation} \label{E:Eq-9}
\frac{t}{2}\,(\varphi ^{2})^{\prime }+ \varphi^{2} + 1= 2at\sqrt{\varphi ^{2}+1}.
\end{equation}%
Now, setting $z(t)=\sqrt{\varphi ^{2}(t)+1}$, we transform equation \eqref{E:Eq-9} into the next equation
\begin{equation*}
z^{\prime }(t)+\frac{1}{t}\,z(t)= 2a,
\end{equation*}%
whose general solution is given by the formula $z(t)=%
\frac{c+ at^{2}}{t}$, $c=const$. Hence, by calculations we obtain that the general solution of  \eqref{E:Eq-9} is 
\begin{equation*}
\varphi (t)=\pm \frac{1}{t}\sqrt{(c+ a\,t^{2})^{2}-t^{2}}.  
\end{equation*}%
In this case, the mean curvature vector field is $H = -a\, N_2$, which implies that $\langle H,H\rangle =a^{2}=const$,  
so $\M_m$ is a CMC-surface. It can be calculated that the Gauss curvature in this case is expressed as $K =\frac{a^2 t^4 - c^2}{t^4}$. Consequently, 
 $\M_m$  is a non-flat CMC-surface lying in a constant hyperplane of $\R^4_1$.

\vskip 2mm Case (ii): $\varkappa \neq 0$, and hence, from the second equality of \eqref{E:Eq-7} we obtain that $f(u) = a$, $a = const \neq 0$.
Using that $\dot{f}^2-\dot{g}^2 = -1$, we get $g(u) = \pm u+b$, $b = const$. We calculate that, in this case, the mean curvature vector field is expressed as follows 
\begin{equation*}
H = \frac{\varkappa}{2a}\, n_1 - \frac{1}{2a} \, n_2,
\end{equation*}
which implies that $\langle H, H \rangle = \frac{1+\varkappa^2}{4a^2}=const$. Hence, the surface $\M_m$  has constant mean curvature. 
Moreover, in this case we get  $K=0$, i.e.  $\M_m$  is a flat surface. 
We consider the normal vector fields 
$$\bar{N}_1 = \frac{1}{\sqrt{\varkappa^2+1}} (N_1 \pm \varkappa\,N_2); \quad \bar{N}_2 = \frac{1}{\sqrt{\varkappa^2+1}} (\mp \varkappa\,N_1 + N_2).$$ 
Using \eqref{E:Eq-mer-derivative-1} we obtain $\widetilde{\nabla}_X \bar{N}_1 = \widetilde{\nabla}_Y \bar{N}_1 = 0$, which implies that $\M_m$ lies in the constant 3-dimensional space $\R^3_1 = \span \{X,Y,\bar{N}_2\}$. 
Finally, we obtain that the meridian surface  $\M_m$ is a flat CMC-surface lying in a hyperplane of $\mathbb{E}^4_1$.

\vskip 1mm Conversely, if one of the cases (i) or (ii) stated in the theorem
holds true, then by direct computations we get that $D_X H = D_YH = 0$, i.e.
the surface has parallel mean curvature vector field.

\end{proof}

\begin{cor}
There are no timelike meridian surfaces of elliptic type with parallel mean curvature vector field other than CMC-surfaces lying in a hyperplane of $\R^4_1$.
\end{cor}

\section{Timelike meridian surfaces of elliptic type with parallel normalized mean curvature vector field}

In this section we will describe the class of timelike meridian surfaces of elliptic type  with parallel normalized mean curvature vector field, but non-parallel $H$. 

Let $\M_m$ be a timelike  meridian surface of elliptic type, defined by \eqref{Eq:meridian_surf}. We assume that $\langle H, H \rangle \neq 0$ and  $H$ is not parallel. The normalized mean curvature vector field is $H_0 =\frac{H}{\sqrt{\langle H, H \rangle}}$. Using  \eqref{E:Eq-mer-H} we get that
\begin{equation}  \label{E:Eq-H0}
H_0 = \frac{\varkappa \sqrt{\dot{f}^2 + 1}}{\sqrt{\varkappa^2 (\dot{f}^2 + 1)^2 + (f \ddot{f}+ \dot{f}^2 + 1)^2}} \,N_1 - \frac{f \ddot{f}+ \dot{f}^2 + 1}{\sqrt{\varkappa^2 (\dot{f}^2 + 1)^2 + (f \ddot{f}+ \dot{f}^2 + 1)^2}} \,N_2.
\end{equation}

The surface $\M_m$ has  parallel normalized mean curvature vector field if and only if 
 $D_X H_0 = D_Y H_0 = 0$.  
Using \eqref{E:Eq-mer-derivative-1} and \eqref{E:Eq-H0} we calculate $D_X H_0$ and $D_Y H_0$ and obtain the following formulas:
\begin{equation*}\label{Eq:NormalizedMeanCurvDeriv}
\begin{array}{l}
\vspace{.2cm}
D_X H_0 = \frac{\partial}{\partial u} \!\left(\!\! \frac{\varkappa \sqrt{\dot{f}^2 + 1}}{\sqrt{\varkappa^2 (\dot{f}^2 + 1)^2 \!+ \!(f \ddot{f}+ \dot{f}^2 + 1)^2}}  \!\! \right) \! N_1 - \frac{\partial}{\partial u} \!\left(\! \!\frac{f \ddot{f}+ \dot{f}^2 + 1}{\sqrt{\varkappa^2 (\dot{f}^2 + 1)^2 \!+ \!(f \ddot{f}+ \dot{f}^2 + 1)^2}} \!\!\right) \! N_2; \\
\vspace{.2cm} 
D_Y H_0 = \frac{1}{f} \frac{\partial}{\partial v} \!\left(\!\! \frac{\varkappa \sqrt{\dot{f}^2 + 1}}{\sqrt{\varkappa^2 (\dot{f}^2 + 1)^2 \!+ \!(f \ddot{f}+ \dot{f}^2 + 1)^2}}  \!\! \right) \! N_1 - \frac{1}{f}\frac{\partial}{\partial v} \!\left(\! \!\frac{f \ddot{f}+ \dot{f}^2 + 1}{\sqrt{\varkappa^2 (\dot{f}^2 + 1)^2 \!+ \!(f \ddot{f}+ \dot{f}^2 + 1)^2}}\! \!\right) \! N_2.
\end{array}
\end{equation*}

These formulas imply that  $\M_m$ has  parallel normalized mean curvature vector field if and only if 
\begin{equation}\label{Eq:NormalizedMeanCurv-cond}
\begin{array}{l}
\vspace{.2cm}
 \frac{\varkappa \sqrt{\dot{f}^2 + 1}}{\sqrt{\varkappa^2 (\dot{f}^2 + 1)^2 \!+ \!(f \ddot{f}+ \dot{f}^2 + 1)^2}}  =const = \alpha; \\
\vspace{.2cm} 
\frac{f \ddot{f}+ \dot{f}^2 + 1}{\sqrt{\varkappa^2 (\dot{f}^2 + 1)^2 \!+ \!(f \ddot{f}+ \dot{f}^2 + 1)^2}}  =const =\beta,
\end{array}
\end{equation}
for some constants $\alpha$ and $\beta$.

\begin{thm}\label{Th:meridSurfParallelNorm}
Let $\M_m$ be a timelike  meridian surface of elliptic type, defined by \eqref{Eq:meridian_surf}. Then,  $\M_m$ has parallel normalized mean curvature vector field (but non-parallel mean curvature vector), if and only if one of the following cases holds:

\hskip 10mm (i)  $\varkappa \neq 0$ and the  meridian $m$ is defined by
$$f(u) = \pm \sqrt{-u^2 +2au+b}, \quad g(u) = \pm \sqrt{a^2+b} \arcsin \frac{u-a}{\sqrt{a^2+b}} + c,$$ 
where $a = const$, $b=const$, $c = const$.

\hskip 10mm (ii)  the curve $\textbf{c}$   has non-zero constant spherical curvature  and the meridian $m$ is determined by $\dot{f} = \varphi(f)$ where
\begin{equation} \notag
\varphi(t) = \pm \frac{1}{t} \sqrt{(c t+a)^2 -t^2}, \quad a =
const, \; c = const \neq 0,\; c^2 \neq \varkappa^2,
\end{equation}
$g(u)$ is defined by $\dot{g} = \pm \sqrt{\dot{f}^2+1}$.
\end{thm}

\begin{proof}
Let $\M_m$ be a timelike  meridian surface of elliptic type with parallel normalized mean curvature vector field, i.e. $D_X H_0 = D_Y H_0 =0$. 
Then,  equalities  \eqref{Eq:NormalizedMeanCurv-cond} hold true. We will consider the following cases.

\vskip 1mm
Case (i):  $f \ddot{f} + \dot{f}^2 + 1= 0$. In this case, from \eqref{E:Eq-H0} it follows that the normalized mean curvature vector field is $H_0 = n_1$ and the mean curvature vector field is 
$H = \frac{\varkappa}{2f} \,n_1$. Since we are interested in surfaces satisfying  $\langle H,H \rangle \neq 0$, we get that $\varkappa \neq 0$. The general solution of the differential equation 
$f \ddot{f} + \dot{f}^2 + 1= 0$ is 
$$f(u) = \pm \sqrt{-u^2 +2au+b},$$
 where $a=const$, $b =const$. Having in mind that  $\dot{g} = \pm \sqrt{\dot{f}^2+1}$, after integration we obtain the following expression for $g(u)$:
$$g(u) = \pm \sqrt{a^2+b} \arcsin \frac{u-a}{\sqrt{a^2+b}} + c, \quad c = const.$$

\vskip 1mm
Case (ii):  $f \ddot{f} + \dot{f}^2 + 1 \neq 0$ in a sub-interval $I_0 \subset I \subset \R$. In this case, equalities \eqref{Eq:NormalizedMeanCurv-cond} imply
\begin{equation} \label{E:Eq-13}
\frac{\beta}{\alpha} \,\varkappa =  \frac{f \ddot{f} + \dot{f}^2 + 1}{\sqrt{\dot{f}^2+1}},  \quad \alpha \neq 0, \; \beta \neq 0.
\end{equation}
The left-hand side of \eqref{E:Eq-13} is a function of the parameter $v$, while the right-hand side of \eqref{E:Eq-13} is a function of the parameter $u$. Hence, we obtain that 
\begin{equation} \notag
\begin{array}{l}
\vspace{2mm}
\frac{f \ddot{f} + \dot{f}^2 + 1}{\sqrt{\dot{f}^2+1}} = c, \quad c = const \neq 0;\\
\varkappa = \frac{\alpha}{\beta}\, c = const.
\end{array}
\end{equation}
Now, the length of the mean curvature vector field is $\langle H,H \rangle  = \frac{\varkappa^2 + c^2}{4f^2} \neq 0$. 
In this case, the meridian $m$ is determined by the following differential
equation:
\begin{equation} \label{E:Eq-14}
f \ddot{f} + \dot{f}^2 + 1 = c \sqrt{\dot{f}^2+1}.
\end{equation}
To find the solutions we set $\dot{f} = \varphi (f)$ in equation \eqref{E:Eq-14} and obtain
that the function $\varphi  = \varphi (t)$ satisfies
\begin{equation} \label{E:Eq-15}
\frac{t}{2} \,(\varphi ^2)' + \varphi ^2 + 1 = c \sqrt{\varphi ^2 + 1}.
\end{equation}
Setting $z(t) = \sqrt{\varphi ^2(t) +1}$, we obtain the equation 
\begin{equation} \notag
z' + \frac{1}{t}\, z = \frac{c}{t},
\end{equation}
whose solution  is given by the formula $z(t) = \frac{ct + a}{t}$, $a = const$.
Consequently, the general solution of \eqref{E:Eq-15} is given by the formula
\begin{equation} \notag
\varphi(t) = \pm \frac{1}{t} \sqrt{(c t+a)^2 -t^2}.
\end{equation}

Conversely, if one of the cases (i) or (ii)  in the theorem  holds true, then by direct computation one can obtain  that
$D_X H_0 = D_Y H_0 = 0$, which means that the surface has parallel normalized  mean curvature vector field.
Moreover, in case (i) we calculate that
$$D_XH = - \frac{\varkappa \dot{f}}{2f^2} \,N_1; \qquad D_YH = \frac{\varkappa'}{2f^2} \,N_1,$$
and in case (ii) we get
$$D_XH = -\frac{\varkappa \dot{f}}{2f^2} \,N_1 - \frac{c\dot{f}}{2f^2} \,N_2; \qquad D_YH = 0,$$
which imply that in both cases $H$ is not parallel in the normal bundle, since $\varkappa \neq 0$, $\dot{f} \neq 0$. 

\end{proof}

We will use the result in Theorem \ref{Th:meridSurfParallelNorm} to find explicit solutions to the system of PDEs \eqref{E:Eq-fund-1}.

\section{Examples}

In this section we will give examples of solutions to the background systems of natural partial differential equations describing
the timelike surfaces with parallel normalized mean curvature vector field in $\R^4_1$.

\vskip 1mm

To apply the theory of timelike surfaces  in the Minkowski 4-space $\R^4_1$ developed in \cite{B-M-2} and \cite{B-M} to the class of meridian surfaces of elliptic type, we need to find the geometric  isotropic frame field $\{x,y,n_1,n_2\}$ introduced in the Preliminaries and the isotropic parametrization. 

Let us consider a  timelike meridian surface $\M_m$  determined by  a spherical curve $\textbf{c}$ with curvature $\varkappa \neq 0$ and
  meridian curve $m: u \rightarrow (f(u), g(u))$.
To introduce isotropic parameters on the surface we consider the following change of the parametrization:
$$
\left |
\begin{array}{l} \ds
\vspace{.3cm}
\bar{u} = \frac{1}{\sqrt{2}} \int{\frac{1}{f(u)}} d u + \frac{v}{\sqrt{2}} \\
\ds \bar{v} = \frac{1}{\sqrt{2}} \int{\frac{1}{f(u)}} d u - \frac{v}{\sqrt{2}}
\end{array}
\right .
$$
Then, the derivatives $z_{\bar{u}}, z_{\bar{v}}$ with respect to the new parameters $\bar{u}, \bar{v}$ are:
$$
\begin{array}{l} 
\vspace{.2cm}
\ds z_{\bar{u}} = \frac{f}{\sqrt{2}} z_u + \frac{1}{\sqrt{2}} z_v; \\
\vspace{.2cm}
\ds z_{\bar{v}} = \frac{f}{\sqrt{2}} z_u - \frac{1}{\sqrt{2}} z_v,
\end{array}
$$
and hence, the coefficients of the first fundamental form in terms of the new parameters are:
$$
\begin{array}{l} 
\vspace{.2cm}
\bar{E} = \langle z_{\bar{u}} , z_{\bar{u}} \rangle = 0; \\
\vspace{.2cm}
\bar{F} = \langle z_{\bar{u}} , z_{\bar{v}} \rangle = -f^2(\bar{u}, \bar{v}); \\
\vspace{.2cm}
\bar{G} = \langle z_{\bar{v}} , z_{\bar{v}} \rangle = 0,
\end{array}
$$
which show that  $(\bar{u}, \bar{v})$ are isotropic parameters of the surface. We will introduce the geometric frame field in the sense of \cite{B-M-2} and will find the geometric functions of the meridian surface when $\M_m$ has  parallel normalized mean curvature vector field (but non-parallel mean curvature vector). 

The isotropic frame field of the tangent space is determined by: 
\begin{equation*} \label{E:Eq-mer-isotropic}
\begin{array}{l}
\vspace{2mm}
x = \ds \frac{z_{\bar{u}}}{f} = \frac{X+Y}{\sqrt{2}}; \\
\vspace{2mm}
y = \ds \frac{z_{\bar{v}}}{f} = \frac{X-Y}{\sqrt{2}}.
\end{array} 
\end{equation*}
Obviously, 
$\langle x, x \rangle = 0, \,\, \langle x, y \rangle = -1, \,\,\langle y, y \rangle = 0$.
We consider the normal vector fields $n_1$ and $n_2$, defined by:
$$
\begin{array}{l}
\vspace{2mm}
n_1 = \ds \frac{\varkappa \sqrt{\dot{f}^2 + 1}}{\sqrt{\varkappa^2(\dot{f}^2 + 1) + (1 + \dot{f}^2 + f \ddot{f})^2}} \,N_1 - \frac{1 + \dot{f}^2 + f \ddot{f}}{\sqrt{\varkappa^2(\dot{f}^2 + 1) + (1 + \dot{f}^2 + f \ddot{f})^2}} \,N_2, \\
\vspace{2mm}
n_2 = \ds \frac{1 + \dot{f}^2 + f \ddot{f}}{\sqrt{\varkappa^2(\dot{f}^2 + 1) + (1 + \dot{f}^2 + f \ddot{f})^2}} \,N_1 + \frac{\varkappa \sqrt{\dot{f}^2 + 1}}{\sqrt{\varkappa^2(\dot{f}^2 + 1) + (1 + \dot{f}^2 + f \ddot{f})^2}} \,N_2.
\end{array}
$$
Then, the frame field $\{x, y, n_1, n_2\}$ is the geometric frame field of  $\M_m$ in the sense of \cite{B-M-2}, since $x$ and $y$ are parallel to the lightlike directions in the tangent space of the surface and $n_1$ is parallel to the mean curvature vector field $H$.

Now, we will consider meridian surfaces with parallel normalized mean curvature vector field (but non-parallel mean curvature vector). According to Theorem \ref{Th:meridSurfParallelNorm} we have two cases: (i) and (ii).

In the first case, $m$ is  defined by: 
$$f(u) = \pm \sqrt{-u^2 +2au+b}, \qquad g(u) = \pm \sqrt{a^2+b} \arcsin \frac{u-a}{\sqrt{a^2+b}} + c,$$ 
 where $a = const$, $b=const$, $c = const$. In this case, by long but standard calculations using the geometric frame field, formulas \eqref{E:Eq-mer-derivative-1} and \eqref{E:DerivFormIsotr},  we obtain that the geometric functions of the surface are expressed as follows: 
\begin{equation*}\label{Eq:MeridianInvariants-0}
\begin{array}{l}
\vspace{.2cm}
\gamma_1 = - \gamma_2 =  \ds\frac{u-a}{\sqrt{2}(-u^2+2a u + b)};\\
\vspace{.2cm}
\nu = \lambda_1 = \lambda_2 = \ds \pm \frac{ \varkappa}{2\sqrt{-u^2+2a u + b}};\\
\vspace{.2cm}
\mu_1 = \mu_2 =  \ds\frac{-\sqrt{a^2+b}}{-u^2+2a u + b};\\
\vspace{.2cm}
\beta_1 = \beta_2 = 0.\\
\end{array}
\end{equation*}

In the second case, using that $f \ddot{f} + \dot{f}^2 + 1 = c \sqrt{\dot{f}^2+1}$, we obtain the following expressions for the geometric functions of the surface:
\begin{equation}\label{Eq:MeridianInvariants9f}
\begin{array}{l}
\vspace{.2cm}
\gamma_1 = - \gamma_2 =\ds  -\frac{\dot{f}}{\sqrt{2}f};\\
\vspace{.2cm}
\nu = \ds \frac{\sqrt{\varkappa^2 +c^2}}{2f};\\
\vspace{.2cm}
\lambda_1 = \lambda_2 =\ds  \frac{\varkappa^2 -c^2 + 2 c \sqrt{\dot{f}^2 + 1}}{2f \sqrt{\varkappa^2 +c^2}};\\
\vspace{.2cm}
\mu_1 = \mu_2 =\ds  \frac{\varkappa \left(c-\sqrt{\dot{f}^2 + 1}\right)}{f \sqrt{\varkappa^2 +c^2}};\\
\vspace{.2cm}
\beta_1 = \beta_2 = 0.
\end{array}
\end{equation}

Note that, in formulas \eqref{Eq:MeridianInvariants9f}, $\varkappa$ is a non-zero constant and $f$ is a function determined by  $\dot{f} = \varphi(f)$, where
$\varphi(t) = \pm \frac{1}{t} \sqrt{(c t+a)^2 -t^2}, \quad a = const, \; c = const \neq 0,\; c^2 \neq \varkappa^2$. 

\vskip 2mm
The geometric functions in the two cases considered above give solutions to the system of PDEs  \eqref{E:Eq-fund-1}. Indeed, let $\varkappa (v)$ be an arbitrary non-zero function and consider the functions 
\begin{equation}  \label{E:Eq-nn}
\begin{array}{l}
\vspace{2mm}
\lambda (u,v) = \ds \frac{ \varkappa}{2\sqrt{-u^2+2a u + b}}; \\
\vspace{2mm}
\mu (u,v) = \ds\frac{-\sqrt{a^2+b}}{-u^2+2a u + b}; \\
\vspace{2mm}
\nu (u,v) = \ds  \frac{ \varkappa}{2\sqrt{-u^2+2a u + b}}. 
\end{array}
\end{equation}
Changing the parameters $(u,v)$ with
\begin{equation*} 
\begin{array}{l}
\vspace{2mm}
\bar{u} = \frac{1}{\sqrt{2}} \arcsin{\frac{u-a}{\sqrt{a^2 +b}}} + \frac{v}{\sqrt{2}};\\ 
\vspace{2mm}
\bar{v} = \frac{1}{\sqrt{2}} \arcsin{\frac{u-a}{\sqrt{a^2 +b}}} - \frac{v}{\sqrt{2}},\\
\end{array}
\end{equation*}
we obtain functions  $\lambda(u(\bar{u},\bar{v}), v(\bar{u},\bar{v}))$, $\mu(u(\bar{u},\bar{v}), v(\bar{u},\bar{v}))$,  $\nu(u(\bar{u},\bar{v}), v(\bar{u},\bar{v}))$,  defined by \eqref{E:Eq-nn}, 
which give a solution to the  background system of partial differential equations
\begin{equation} \label{E:Eq-syst1}
\begin{array}{l}
\vspace{2mm}
\nu_{\bar{u}} + \lambda_{\bar{v}} = \lambda (\ln|\mu|)_{\bar{v}};\\
\vspace{2mm}
\lambda_{\bar{u}} + \nu_{\bar{v}} = \lambda (\ln|\mu|)_{\bar{u}};\\
\vspace{2mm}
|\mu| \left(\ln |\mu|\right)_{{\bar{u}}{\bar{v}}} = \lambda^2 + \mu^2 - \nu^2.
\end{array} 
\end{equation}
This is system  \eqref{E:Eq-fund-1} in the case $\varepsilon = -1$, since it can easily be checked that in case (i) the meridian surface  $\M_m$ satisfies $K - H^2 <0$.

\vskip 1mm
Although we obtained this solution to \eqref{E:Eq-syst1} in a geometric way (using geometric construction), one can check also by direct computations that the functions  $\lambda(u(\bar{u},\bar{v}), v(\bar{u},\bar{v}))$, $\mu(u(\bar{u},\bar{v}), v(\bar{u},\bar{v}))$,  $\nu(u(\bar{u},\bar{v}), v(\bar{u},\bar{v}))$,  defined by \eqref{E:Eq-nn} satisfy the equations in system \eqref{E:Eq-syst1}.

\vskip 1mm

So, studying the class of timelike meridian surfaces of elliptic type, we find a whole family of solutions to the background system of PDEs describing the timelike surfaces with parallel normalized mean curvature vector field in $\R^4_1$. Given an arbitrary non-zero function  $\varkappa(v)$ and arbitrary constants $a$, $b$, $c$, by \eqref{E:Eq-nn} we obtain a solution to system \eqref{E:Eq-syst1}.

\vskip 3mm
\noindent
\textit{Example 1.} 
Let $a=1$, $b=3$ and consider the functions 
\begin{equation} \label{E:Eq-ex1}
\begin{array}{l}
\vspace{2mm}
\lambda (u,v) = \ds \frac{\varkappa (v)}{2\sqrt{-u^2+ 2 u + 3}}; \\
\vspace{2mm}
\mu (u,v) = \ds\frac{2}{u^2 -2 u - 3}; \\
\vspace{2mm}
\nu (u,v) = \ds  \frac{\varkappa(v)}{2\sqrt{-u^2+ 2 u + 3}}, 
\end{array}
\end{equation}
where $u\in(-1; 3)$ and $\varkappa (v) \neq 0$ is an arbitrary function. By direct computations it can be checked that changing  the parameters  $(u,v)$ with
\begin{equation*} 
\begin{array}{l}
\vspace{2mm}
\bar{u} = \frac{1}{\sqrt{2}} \arcsin{\frac{u-1}{2}} + \frac{v}{\sqrt{2}};\\ 
\vspace{2mm}
\bar{v} = \frac{1}{\sqrt{2}} \arcsin{\frac{u-1}{2}} - \frac{v}{\sqrt{2}},\\
\end{array}
\end{equation*}
functions \eqref{E:Eq-ex1} satisfy system \eqref{E:Eq-syst1}.

\vskip 3mm
\noindent
\textit{Example 2.} 
 Let $a=5$, $b=0$. Now, we consider the functions 
\begin{equation} \label{E:Eq-ex2}
\begin{array}{l}
\vspace{2mm}
\lambda (u,v) = \ds \frac{\varkappa (v)}{2\sqrt{u(10 -u)}}; \\
\vspace{2mm}
\mu (u,v) = \ds\frac{5}{u(u -10)}; \\
\vspace{2mm}
\nu (u,v) = \ds  \frac{\varkappa(v)}{2\sqrt{u(10 -u)}}, 
\end{array}
\end{equation}
where $u\in(0; 10)$ and $\varkappa (v) \neq 0$ is an arbitrary function. Changing  the parameters  $(u,v)$ with
\begin{equation*} 
\begin{array}{l}
\vspace{2mm}
\bar{u} = \frac{1}{\sqrt{2}} \arcsin{\frac{u-5}{5}} + \frac{v}{\sqrt{2}};\\ 
\vspace{2mm}
\bar{v} = \frac{1}{\sqrt{2}} \arcsin{\frac{u-5}{5}} - \frac{v}{\sqrt{2}},\\
\end{array}
\end{equation*}
we get that the functions  $\lambda(u(\bar{u},\bar{v}), v(\bar{u},\bar{v}))$, $\mu(u(\bar{u},\bar{v}), v(\bar{u},\bar{v}))$,  $\nu(u(\bar{u},\bar{v}), v(\bar{u},\bar{v}))$,  defined by \eqref{E:Eq-ex2} give a solution to system \eqref{E:Eq-syst1}.

\vskip 2mm
Solutions to the system of PDEs describing the timelike surfaces with parallel normalized mean curvature vector field in $\R^4_1$ can be found also in the class of timelike meridian surfaces of  hyperbolic or parabolic type. They are constructed as one-parameter systems of meridians of a rotational hypersurface with spacelike or lightlike axis. These classes of timelike surfaces will be studied in a separate paper.

\vskip 6mm 
\textbf{Acknowledgments:}
The first author is supported by the National Scientific Programme ''Young Scientists and Post-Doctoral Students -- 2''. The second author is partially supported by the National Science Fund, Ministry of Education and Science of Bulgaria under contract KP-06-N82/6.


\vskip 6mm




\begin{thebibliography}{99}

\bibitem{AGM} 
Aleksieva Y., Ganchev G., Milousheva V., \textit{On the theory of Lorentz surfaces with parallel normalized mean curvature vector field in pseudo-Euclidean 4-space}. J. Korean Math. Soc. \textbf{53} (2016), no. 5, 1077--1100.

\bibitem{A-M-1}
Aleksieva Y., Milousheva V., \textit{Minimal Lorentz surfaces in pseudo-Euclidean 4\nobreakdash-space with neutral metric}.
J. Geom. Phys., \textbf{142}, (2019), 240--253.

\bibitem{Al-Pal}
Al\'{i}as L., Palmer B., \textit{Curvature properties of zero mean curvature surfaces in four dimensional Lorenzian space forms}. Math. Proc. Cambridge Philos. Soc. \textbf{124} (1998), 315--327.

\bibitem{AM}
Arslan K., Milousheva V., \textit{Meridian surfaces of elliptic or hyperbolic type with pointwise 1-type Gauss map in Minkowski 4-space}, Taiwanese J. Math., \textbf{20}, no. 2 (2016), 311--332.

\bibitem{B-M}
Bencheva V., Milousheva V., \textit{Timelike Surfaces with Parallel Normalized Mean Curvature Vector Field}, Turkish J. Math. (2024), Vol. 48: no. 2, Article 15.

\bibitem{B-M-2}
Bencheva V., Milousheva V., \textit{Fundamental Theorems for Timelike Surfaces in the Minkowski 4-Space}, C. R. Acad. Bulg. Sci., no. 2 (2024), Vol. 77, 167-178.

\bibitem{Chen-MM} Chen B.-Y., \textit{Surfaces with parallel normalized mean
curvature vector}, Monatsh. Math., \textbf{90}, no. 3 (1980), 185--194.

\bibitem{Chen}
Chen B.-Y., \textit{Pseudo-Riemannian geometry, $\delta$-invariants and
applications}. World Scientific Publishing Co. Pte. Ltd.,
Hackensack, NJ, 2011.

\bibitem{GM6}
 Ganchev G.,  Milousheva V.,  \emph{An Invariant Theory of Marginally
Trapped Surfaces in the Four-dimensional Minkowski Space}, J. Math. Phys. \textbf{53}  (2012) Article ID: 033705, 15 pp.

\bibitem{G-M-IJM}
Ganchev G.,  Milousheva V., \textit{Timelike surfaces with zero mean curvature in Minkowski 4-space}, Israel J. Math. 
\textbf{196} (2013), 413--433.

\bibitem{GM-MC}
Ganchev G., Milousheva V., \emph{Meridian Surfaces of Elliptic or Hyperbolic Type in the
Four-dimensional Minkowski Space}, Math. Commun., \textbf{21}, no. 1 (2016),  1--21.

\bibitem{G-M-Fil}
Ganchev G.,  Milousheva V.,  \textit{Surfaces with parallel normalized mean curvature vector field in Euclidean or Minkowski 4-space}, Filomat Vol. 33, no. 4 (2019), 1135--1145.

\bibitem{Itoh}
Itoh T., M\textit{inimal surfaces in 4-dimensional Riemannian manifolds of constant curvature}, Kodai
Math. Sem. Rep., 23 (1971), 451--458.


\bibitem{Lar}
Larsen J.C., \textit{Complex analysis, maximal immersions and metric singularities}, Monatsh. Math. \textbf{122} (1996), 105--156.

\bibitem{Lund-Reg}
Lund, F., T. Regge, \textit{Unified approach to strings and vortices with soliton solutions}. Phys. Rev.  D, 14, no. 6 (1976), 1524--1536.

\bibitem{O'N}
O'Neill M., \textit{Semi-Riemannian geometry with applications to relativity},
Academic Press, London 1983.

\bibitem{Sa}
Sakaki M., \textit{Lorentz stationary surfaces in 4-dimensional space forms of index 2}, Tsukuba J. Math. \textbf{35} (2011), 2,  215--229.

\bibitem{Sen}
S\c en R. \textit{Biconservative submanifolds with parallel normalized mean curvature vector field in Euclidean spaces}. Bull. Iran. Math. Soc. \textbf{48 } (2022), 3185--3194. 

\bibitem{Sen-Turg-JMAA}
S\c en R, Turgay NC. \textit{On biconservative surfaces in 4-dimensional Euclidean space}. J. Math. Anal. Appl. \textbf{460} (2018); 2, 565--581. 

\bibitem{Shu}
Shu, S., \textit{Space-like submanifolds with parallel normalized mean curvature vector field in de Sitter space}.
J. Math. Phys. Anal. Geom.  \textbf{7 } (2011),  no. 4, 352--369.


\bibitem{Trib-Guad}
Tribuzy R., Guadalupe I., \textit{Minimal immersions of surfaces into 4-dimensional space forms}, Rend.
Sem. Mat. Univ. Padova, 73 (1985), 1--13.


\end{thebibliography}
\end{document}